\documentclass{article}

\title{The Gr\"{o}bner Basis of the\\ Ideal of Vanishing Polynomials}

\author{Gert-Martin Greuel\thanks{Department of Mathematics, Working Group {\it Algebra, Geometry, and Computer Algebra}, University of Kaiserslautern, Germany. Email: greuel@mathematik.uni-kl.de.}, Frank Seelisch\thanks{Email: seelisch@mathematik.uni-kl.de.}, Oliver Wienand}

\date{January 2009}

\usepackage{algorithm}
\usepackage{algorithmic}

\usepackage{amsmath}
\usepackage{amsthm}
\usepackage{amssymb}
\usepackage{makeidx}
\usepackage{xspace}

\newcommand{\ideal}[1]{{\left\langle{#1}\right\rangle}}

\newcommand{\ie}{i.\,e.\xspace}

\newcommand{\divz}{|_{\,\atop\!\Z}}
\newcommand{\divm}{|_{\,\atop\!\!\!\: m}}
\newcommand{\ndivz}{\nmid_{\,\atop\!\Z}}

\newcommand{\N}{\mathbb{N}}
\newcommand{\Z}{\mathbb{Z}}

\newcommand{\polyfct}[1]{\widetilde{#1}}
\newcommand{\vanideal}{\ensuremath{{I_0}}}
\newcommand{\vangb}[1][]{\ensuremath{{G_0}}}
\newcommand{\defemph}[1]{\textbf{#1}\index{#1}}
\newcommand{\degree}[1]{\deg\left({#1}\right)}

\newcommand{\skalar}[1]{\left\langle #1\right\rangle}
\newcommand{\card}[1]{{\left\lvert #1\right\rvert}}

\DeclareMathOperator{\lt}{LT} %
\DeclareMathOperator{\lc}{LC} %
\DeclareMathOperator{\lm}{LM} %
\DeclareMathOperator{\li}{L} %
\DeclareMathOperator{\gcod}{gcd} %
\newcommand{\wo}{\backslash} %
\newcommand{\LT}[1]{\lt\left( #1\right)} %
\newcommand{\LM}[1]{\lm\left( #1\right)} %
\newcommand{\LC}[1]{\lc\left( #1\right)} %
\newcommand{\LI}[1]{\li\left( #1\right)} %
\newcommand{\ggt}[1]{\gcod\left( #1\right)} %

\newcommand{\cR}{C}
\newcommand{\wR}{{\cR[\vars]}}

\newcommand{\vI}{\vanideal}

\newcommand{\vars}{\mathbf{x}}
\newcommand{\monom}{\mathbf{x}}

\newcommand{\Zmod}[1]{\Z/{#1}}

\newtheorem{thm}{Theorem}[section]
\newtheorem{prop}[thm]{Proposition}
\newtheorem{defn}[thm]{Definition}

\newtheorem{lem}[thm]{Lemma}

\newtheorem{cor}[thm]{Corollary}

\newtheorem{example}[thm]{Example}

\makeindex

\begin{document}

\maketitle

\begin{abstract}
We construct an explicit minimal strong Gr\"{o}bner basis of the ideal of vanishing polynomials in the polynomial ring over $\Zmod{m}$ for $m\geq 2$. The proof is done in a purely combinatorial way. It is a remarkable fact that the constructed Gr\"{o}bner basis is independent of the monomial order and that the set of leading terms of the constructed Gr\"{o}bner basis is unique, up to multiplication by units. We also present a fast algorithm to compute reduced normal forms, and furthermore, we give a recursive algorithm for building a Gr\"{o}bner basis in $\Zmod{m}[x_1,x_2,\ldots,x_n]$ along the prime factorization of $m$. The obtained results are not only of mathematical interest but have immediate applications in formal verification of data paths for microelectronic systems-on-chip.
\end{abstract}

\section{Introduction}
Although the basic properties of Gr\"{o}bner bases in polynomial rings over a ring $\cR$ are well-known (see \cite{adams}), they have not been studied very much, mainly because they were considered as academic, in contrast to the case where the ground ring $\cR$ is a field. Recently however, Gr\"{o}bner basis techniques in polynomial rings over $\cR=\Zmod{m}$ (in particular $\Zmod{2^k}$) have attracted some attention due to their potential applications to proving correctness of data paths in system-on-chip design (cf.~e.g.~\cite{newdevel},\cite{equi_check}, \cite{datacorr}).\\
When the underlying ring $\cR$ has only finitely many elements, then there exist polynomials in $\cR[x_1,x_2,\ldots,x_n]$ which evaluate to zero for all $(a_1,a_2,\ldots,a_n)\in\cR^n$, called vanishing polynomials. Thus, any polynomial function $\polyfct{f}:\cR^n\rightarrow\cR$ given by an arbitrary element $f\in\cR[x_1,x_2,\ldots,x_n]$, will have many alternative representations in $\cR[x_1,x_2,\ldots,x_n]$, as $\polyfct{f}=\polyfct{f+g}$, for all $g$ that constantly vanish on $\cR^n$. All vanishing polynomials constitute an ideal $I_0$.\\
In the applications mentioned above, not the polynomials but only the polynomial functions are of interest. Thus, if we want to apply algebraic methods we need to be able to efficiently compute normal forms of polynomials with respect to a Gr\"{o}bner basis of $I_0$. In the presented paper, we set the theoretical ground and provide fast algorithms for doing these computations.\\[1ex]
From a mathematical point of view, $I_0\subset\Zmod{m}[x_1,x_2,\ldots,x_n]$ has some interesting properties. In this paper, we will give an explicit minimal strong Gr\"{o}bner basis $G_m$ for $I_0$. As will turn out, $G_m$ is a Gr\"{o}bner basis with respect to {\it every} global monomial order. Moreover, we will show for any alternative minimal strong Gr\"{o}bner basis $G$ of $I_0\subset\Zmod{m}[x_1,x_2,\ldots,x_n]$ that the sets of leading terms of $G_m$ and $G$ are the same up to multiplication by units. This is
remarkable, since the ring $\Zmod{m}$ has zero divisors. In general, the leading terms of two minimal strong Gr\"{o}bner bases of an ideal $I\subset\cR[x_1, x_2,\ldots, x_n]$ need not be related by a unit but only by some element of $\cR$. We will prove both properties and show also that in general all minimal strong Gr\"{o}bner bases of an arbitrary ideal $I\subset\cR[x_1, x_2,\ldots, x_n]$ have the same number of elements.\\[1ex]
From a practical point of view, as mentioned above, engineering tasks involving the computation of Gr\"{o}bner bases over finite rings will often need to deal with vanishing polynomials. This is due to the fact that normally the elements of a Gr\"{o}bner basis $G$ will be used to decide the consistency of a mathematical model. And typically, such a check involves the question whether the set of zeros of all polynomials $f\in G$ coincides with the set of all feasible input-output vectors of the modelled artifact; see also \cite{newdevel}. Our interest was specifically spurred by a cooperation with the local Electronic Design Automation Group in which we use Gr\"{o}bner bases to formally verify chip designs. More precisely, a given verification task is translated into a polynomial ideal in~$\Zmod{2^k}$, where typically $k=32$ or $k=64$; cf.~\cite{ollisphd}.
For the special case of polynomial datapath verification we
also refer to \cite{datacorr} in which it was shown that the Gr\"{o}bner basis approach proves tractable for industrial applications where standard property checking techniques failed.\\[1ex]
This paper is organized as follows. Section 2 briefly recalls the basic concepts from the theory of polynomial rings and Gr\"{o}bner bases needed later. Section 3 starts by presenting canonical members of the ideal of vanishing polynomials $I_0\subset\Zmod{m}[x_1,x_2,\ldots,x_n]$. Next we show that the leading term of any given vanishing polynomial is divisible by the leading term of an appropriate canonical member. This relation enables us to finally construct an explicit minimal strong Gr\"{o}bner basis $G_m$ of $I_0\subset\Zmod{m}[x_1,x_2,\ldots,x_n]$. We also show that the size of $G_m$ is of polynomial order of degree $k$ in the number of variables $n$, when we are in the practically relevant case $m=2^k$.\\
The theoretical results are followed by algorithms for computing reduced normal forms with respect to the constructed basis, and for recursively computing a Gr\"{o}bner basis of $I_0\subset\Zmod{m}[x_1,x_2,\ldots,x_n]$ along the prime factorization of $m$. The normal form algorithm has been implemented in the computer algebra system {SINGULAR} \cite{singsys} and successfully applied, \cite{datacorr}.

\section{Preliminaries}

Let~$\cR$ be a commutative, noetherian ring with 1, and~$\wR:=\cR[x_1,x_2,\ldots,x_n]$ a multivariate polynomial ring over~$\cR$, where $n\geq 1$. For any multi-index\linebreak $\alpha=(\alpha_1, \ldots, \alpha_n)\in\{0,1,2,\ldots\}^n$, a product of variables~$\monom^\alpha:=x_1^{\alpha_1}\cdots x_n^{\alpha_n}$ is called a monomial, and a product~$a\cdot\monom^\alpha$ with~$a\in\cR$ is called a term.\\
Given two multi-indices $\alpha=(\alpha_1,\ldots,\alpha_n),\beta=(\beta_1,\ldots,\beta_n)$, we define $\alpha\pm\beta :=(\alpha_1\pm\beta_1,\ldots,\alpha_n\pm\beta_n)$. We may compare $\alpha$ and $\beta$ according to the predicate $\alpha\preceq\beta :\Leftrightarrow \forall i\in\{1,\ldots, n\}\: :\:\alpha_i\leq\beta_i$, and similarly $\alpha\prec\beta:\Leftrightarrow\alpha\preceq\beta \wedge\alpha\neq\beta$. For $\alpha=(\alpha_1, \ldots, \alpha_n)\in\{0,1,2,\ldots\}^n$, we write $\alpha!:=\alpha_1!\cdots\alpha_n!$, and $|\alpha|:=\alpha_1+\ldots +\alpha_n$.\\
Moreover, we require the polynomial ring $\wR$ to be equipped with a global monomial order $<$, \ie, $<$ is a well-order on the set of monomials and satisfies $\monom^\alpha > \monom^\beta \Rightarrow\monom^{\alpha+\gamma} > \monom^{\beta+\gamma}$ for all $\alpha,\beta,\gamma\in\{0,1,2,\ldots\}^n$. Then $<$ refines the partial order $\prec$.\\
Since we are going to work with divisibility in $\Zmod{m}[x_1,x_2,\ldots,x_n]$, we need to distinguish between divisibility in $\Zmod{m}$ and in $\Z$. We set $a\divz b:\Leftrightarrow \exists\:k\in\Z\: :$\linebreak $b=a\cdot k$ and $a\divm b:\Leftrightarrow \exists\:k\in\Z\: :\: m\divz (b-a\cdot k)$, that is, $b$ and $a\cdot k$ represent the same residue class in $\Zmod{m}$. For two monomials $a\vars^{\alpha},b\vars^{\beta}$, we say that $a\vars^{\alpha}$ divides $b\vars^{\beta}$, if $a\divm b\wedge \alpha\preceq\beta$. We then write $a\vars^{\alpha}| b\vars^{\beta}$, using the ordinary symbol.\\[1ex]
Let~$f=a_0\cdot \monom^{\alpha^{(0)}}+\dots+a_k\cdot \monom^{\alpha^{(k)}}$ be a polynomial in~$\cR[x_1,x_2,\ldots,x_n]$ with~$a_i\neq 0$ for $0\leq i\leq k,$ and~$x^{\alpha^{(0)}} > x^{\alpha^{(1)}} > \dots > x^{\alpha^{(k)}}$. We use the following notation:
\begin{align*}
\degree{f}&= \max\{\card{\alpha^{(i)}}~|~0\leq i\leq k\}&&\text{total degree of $f,$}\\
\LT{f} &= a_0\cdot \monom^{\alpha^{(0)}}&&\text{leading term of $f,$}\\
\LM{f} &= \monom^{\alpha^{(0)}}&&\text{leading monomial of $f,$}\\
\LC{f} &= a_0&&\text{leading coefficient of $f,$}\\
\LI{A} &= \skalar{\LT{f}~|~f\in A}_{\cR[x_1,x_2,\ldots,x_n]}&&\text{leading ideal of $A,$}\\
&&&\text{for $A$}\subset\cR[x_1,x_2,\ldots,x_n],\:A\neq\emptyset.
\end{align*}
For an ideal $I\subset\cR[x_1,x_2,\ldots,x_n]$ a finite set $G\subset\cR[x_1,x_2,\ldots,x_n]$ is called a \defemph{Gr\"{o}bner basis} of $I$ if
\[G\subset I\text{, and }\LI{I}=\LI{G}.\]
That is, $G$ is a Gr\"{o}bner basis, if the leading terms of $G$ generate the leading ideal of $I$. Note that in general, all 
defined objects depend on the chosen monomial order. Especially, a set $G$ may be a Gr\"{o}bner basis only with respect to a certain monomial order. We also remind the reader that with the given definition, $G$ already generates $I$, cf.~\cite{adams}.\\
$G$ is furthermore called a \defemph{strong Gr\"{o}bner basis} if for any $f\in I\wo\{0\}$ there
exists a polynomial $g\in G$ satisfying $\LT{g}|\LT{f}$. A strong Gr\"{o}bner basis $G$ is called \defemph{minimal strong} if~$\LT{g_1}\,\nmid\,\LT{g_2}$ for all distinct~$g_1, g_2\in G$. It is a well-known fact that a strong Gr\"{o}bner basis can always be constructed from a given Gr\"{o}bner basis when $\cR$ is a principal ideal domain, see e.g.~\cite{adams}.\\
Note that if~$\cR$ is a field, any non-zero coefficient of a term is invertible in $\cR$, and thus $\LI{A} = \skalar{\LM{f}~|~f\in A}$. It is easy to verify that in this case every Gr\"{o}bner basis is a strong Gr\"{o}bner basis. As the following example shows, this does in general not hold when $\cR$ is a ring:
\begin{example}
Consider $\cR := \Zmod{6}$, and the polynomial ring $\cR[x]$ with one variable. Then $G:=\{2x, 3x\}$ is a Gr\"{o}bner basis of the ideal $I:=\ideal{x}$.  But since neither $2x$ nor $3x$ divide $x,\:G$ is not a strong Gr\"{o}bner basis.
\end{example}

We shall now capture the central notions of this paper.

\begin{defn}
To any polynomial~$f\in\cR[x_1,x_2,\ldots,x_n]$ we associate the polynomial function~$\polyfct{f}:\cR^n\to\cR$, $(c_1,c_2,\ldots,c_n)\mapsto f(c_1,c_2,\ldots,c_n)$. We call~$f$ a \defemph{vanishing polynomial} if the function~$\polyfct{f}$ is identically zero.\\
The set $\vI=\{f\in\cR[x_1,x_2,\ldots,x_n]~|~f\text{ is a vanishing polynomial}\}$ is obviously an ideal in $\cR[x_1,x_2,\ldots,x_n]$, called the \defemph{ideal of vanishing polynomials}.
\end{defn}

\section{A Minimal Strong Gr\"{o}bner Basis of the\\
Ideal of Vanishing Polynomials}

\subsection{The Ideal of Vanishing Polynomials}
From now on let the coefficient ring be $\cR=\Zmod{m}$, where $m\geq 2$, except stated otherwise. The following results were inspired by the work of Singmaster \cite{singmaster}, Kempner \cite{kempner}, Halbeisen, Hungerb\"{u}hler, and L{\"a}uchli \cite{halbeisen}, and Hungerb\"{u}hler and Specker \cite{gen_smarand}. Already in Lemma 5 of \cite{kempner}, a univariate version of the following lemma was proven. Theorem 7 of \cite{halbeisen} restated this result, and \cite{gen_smarand} came up with a generalization to multivariate polynomial rings over $\Zmod{m}$.

\begin{lem}\label{lem_zeropoly}
Let~$a\in\Z$ and~$\alpha=(\alpha_1,\dots,\alpha_n)\in\N_0^n$ such that~$m\divz a\alpha!$. Then
\[p_{\alpha,a} := a\prod_{i=1}^n\prod_{l=1}^{\alpha_i}(x_i-l)\in \Zmod{m}[x_1,\dots,x_n]\]
is a vanishing polynomial.
\end{lem}
\begin{proof}
Fix an arbitrary point $(c_1,c_2,\ldots,c_n)\in C^n$. Then $p_{\alpha,a}(c_1,c_2,\ldots,c_n)$ contains, for all $i$, by definition the $\alpha_i$ successive factors $c_i-1,c_i-2,\ldots,$\linebreak $c_i-\alpha_i$. Independent of the value of $c_i$, these contain all factors from $2$ up to $\alpha_i$. Therefore, $\alpha_i!$ divides $p_{\alpha,a}(c_1,c_2,\ldots,c_n)$, for all $i$. By combining these results, it follows immediately that $a\alpha_1!\cdots\alpha_n!$ divides $p_{\alpha,a}(c_1,c_2,\ldots,c_n)$. With $m\divz a\alpha!$ this yields $p_{\alpha,a}(c_1,c_2,\ldots,c_n)=0$ modulo $m$.
\end{proof}

Let us now take a closer look at an arbitrary vanishing polynomial:
\begin{lem}\label{lem_vanishpoly}
Let~$f\in I_0\subset \Zmod{m}[x_1,x_2,\ldots,x_n]$ be an arbitrary vanishing polynomial with $\LT{f}=b\vars^{\beta}$. Then $m\divz b\beta!$.
\end{lem}
For the proof we use some of the ideas introduced in \cite{gen_smarand}, which are based on the notion of partial differences in the multivariate setting. Already Carlitz used partial differences in the univariate case, see \cite{carlitz}, to give a necessary and sufficient condition for a function $f$ over $\Zmod{p^k}$ to be a polynomial function.\footnote{I.e., $f(a)=g(a)$ mod $p^k$, for all $a\in\Zmod{p^k}$ and some polynomial $g\in\Zmod{p^k}[x]$.}
\begin{proof}
Let $\cR[x_1,\ldots,x_n]$ denote an arbitrary polynomial ring over $n\geq 1$ variables, and let $h\in\wR$ be a polynomial. Then we may define the $i^{th}$ partial difference
\[ \nabla_i h:= h(x_1,\ldots,x_{i-1},x_i +1,x_{i+1},\ldots,x_n) - h(x_1,\ldots,x_{i-1},x_i,x_{i+1},\ldots,x_n), \]
for $1\leq i\leq n$. Note that $\nabla_i$ is a linear operator.\\
Now we can define the successive application of the operator by
\[ \nabla_i^0 h:= h,\:\:\:\mbox{and}\:\:\:\nabla_i^{k+1}h:=\nabla_i\nabla_i^k h,\:\:\mbox{for}\:k\geq 0.\]
(For n = 1, $\nabla_1^k h$ coincides with Carlitz' $\triangle^k h$; see \cite{carlitz}.)\\
Since obviously, $\nabla_i\nabla_j h = h(x_1,\ldots,x_i +1,\ldots,x_j +1,\ldots,x_n)-h(x_1,\ldots,$\linebreak $x_i +1,\ldots,x_n)-h(x_1,\ldots,x_j +1,\ldots,x_n)+h(x_1,\ldots,x_n) = \nabla_j\nabla_i h,$ for all $i,j\in\{1,\ldots,n\}$, we can extend the operator to arbitrary multi-indices, that is, with $\alpha=(\alpha_1,\ldots,\alpha_n)\in\{0,1,2,\ldots\}^n$, the term
\[ \nabla^{\alpha}h := \nabla_1^{\alpha_1}\nabla_2^{\alpha_2}\ldots\nabla_n^{\alpha_n} h \]
is independent from the order of application of the $\nabla_i$ operators and hence well-defined.\\
Let us consider the difference $(x_i+1)^k -x_i^k=k\cdot x_i^{k-1}+g(x_i)$, where $g$ consists of lower terms only, that is, $\degree{g}<k-1$. A simple induction shows that $\nabla_i^k x_i^k=k!$ and $\nabla_i^j x_i^k=0$, whenever $j>k$. Let now $a\vars^{\alpha}:=\LT{h}$ denote the leading term. Then, mainly due to the linearity of the $\nabla_i$ operators, it is easy to see that the previous facts can be further abstracted to the general statements
\[ \nabla^{\alpha} h = a\alpha!\:\:\:\mbox{and}\:\:\:\nabla^{\beta} h = 0,\:\:\mbox{for all}\:\:\beta\succ\alpha. \]
We apply the first equation to the vanishing polynomial $f$ over the ring $\Zmod{m}$: With $f$ also $\nabla^{\beta}f=b\beta!$ must be a vanishing polynomial, by construction. But this implies $b\beta!=0$ modulo $m$.
\end{proof}

\subsection{A Minimal Strong Gr\"{o}bner Basis of $I_0$}
The above lemmas suggest to consider the set of all polynomials $p_{\alpha,a}$ for which neither $\alpha$ nor $a$ can be replaced by a smaller multi-index or element of $\Zmod{m}$, respectively, without loosing the condition $m\divz a\alpha!$. (This minimality of $\alpha$ has been inspired by the so-called Smarandache function which maps $m$ to\linebreak $min\{k\in\N\;|\;m\divz k! \}$. This function played a role in previous works which studied the univariate case, and had been named after F. Smarandache, see \cite{smarandache}, although the idea had been introduced earlier by Kempner in Definition 1 of \cite{kempner}.)
We thus define
\begin{align*}
S_m & :=\{\:(\alpha, a) ~|~ 1\leq a<m,\:\:a\divz m,\:\:\alpha\in\N_0^n,\:\:m\divz a\alpha!,\\
    & \mbox{\hspace*{2cm}}\forall\:\beta\prec\alpha\: :\:\:m\ndivz a\beta!,\\
    & \mbox{\hspace*{2cm}}\forall\:b<a,\:b\divz a\: :\:\:m\ndivz b\alpha!\:\},\\
G_m &:= \{\:p_{\alpha,a} ~|~ (\alpha,a)\in S_m\:\}.
\end{align*}
Note that, according to Lemma 3.1, all polynomials in $G_m$ will still be elements of $I_0$. And by Lemma 3.2, we can hope to have constructed a strong Gr\"{o}bner basis. 

\begin{thm}
Let $m\geq 2$ and $n\geq 1$ be arbitrary integers. With the above notations, $G_m$ is a minimal strong Gr\"{o}bner basis of the ideal of vanishing polynomials $I_0\subset\Zmod{m}[x_1,x_2,\ldots,x_n]$, independent of the global monomial order.
\end{thm}
Before we prove the theorem, let us take a look at an example.
\begin{example}
Let $m=q_1\cdot q_2\cdots q_k$ be a product of $k\geq 1$ mutually distinct primes, and $n\geq 1$ arbitrary. We assume $q_1<q_2<\ldots<q_k$. Then
we can immediately write down all elements of $G_m$:
\begin{align*}
& (x_i-1)(x_i-2)\cdots(x_i-q_k),\\
q_k\cdot & (x_i-1)(x_i-2)\cdots(x_i-q_{k-1}),\\
q_k\cdot q_{k-1}\cdot & (x_i-1)(x_i-2)\cdots(x_i-q_{k-2}),\\
& \cdots\\
q_k\cdot q_{k-1}\cdots q_2\cdot & (x_i-1)(x_i-2)\cdots(x_i-q_1),
\end{align*}
in each row for all $i\in\{1,2,\ldots,n\}$.
\end{example}
Note that the first type of polynomial is in $G_m$, as $q_k!$ already contains all $q_j$, thus $m\divz q_k!$. Also, we need to have all $q_k$ polynomial factors since, for all $r<q_k,\:q_k\ndivz r!$, \ie $m\ndivz r!$. For the following polynomials, the argument is similar. Moreover, it is easy to see that we do not have elements in $G_m$ involving two or more variables, and the presented polynomials are all elements of $G_m$.\\
In this special case $|G_m|=k\cdot n$, and the maximal degree is $q_k$. This means that the size of the basis is only linear in the number of variables.\\
For the case $k=1$, $\Zmod{q_1}$ is a field, and we obtain only the $n$ polynomials in the top row, which are well-known for this case.\\
We now prove the theorem:
\begin{proof}
Let us fix $m\geq 2$, the number of variables $n\geq 1$, and an arbitrary global monomial order. We first show that $G_m$ is indeed a Gr\"{o}bner basis of $I_0$. To this end, it suffices to show that (i) $S_m$ and hence $G_m$ is a finite set, (ii) $G_m\subset I_0$, and (iii) $\LI{I_0}\subset\LI{G_m}$, since (ii) implies the other inclusion $\LI{G_m}\subset\LI{I_0}$.\\
(i) Since $(\alpha,a)\in S_m$ implies $\alpha\preceq (m, m,\ldots,m)$, the set is clearly finite.\\
(ii) $G_m$ consists of polynomials $p_{\alpha,a}$ with $m\divz a\alpha!$. Then $G_m\subset I_0$ by Lemma 3.1.\\
(iii) Let $f\in\LI{I_0}$ be arbitrary. Then there exist some integer $N\geq 1,$\linebreak $h_i\in\Zmod(m)[x_1,x_2,\ldots,x_n]$ and $f_i\in I_0,1\leq i\leq N,$ such that
\[ f = \sum_{i=1}^N h_i\cdot\LT{f_i}. \]
Writing $a_i\vars^{\alpha^{(i)}}:=\LT{f_i}$, we obtain $m\divz a_i\alpha^{(i)}!$ from Lemma 3.2. Now either $(\alpha^{(i)},a_i)$ is already an element of $S_m$. Or we can replace $a_i$ by some $b_i\divz a_i$ and/or $\alpha^{(i)}$ by some $\beta^{(i)}\preceq\alpha^{(i)}$ such that $(\beta^{(i)},b_i)\in S_m$. We can subsume both cases in saying that, for each $i\in\{1,2,\ldots,N\}$, there is some $(\beta^{(i)},b_i)\in S_m$ such that $b_i\vars^{\beta^{(i)}}|\LT{f_i}$. With appropriate polynomials $g_i, 1\leq i\leq N,$ this amounts to
\[ f = \sum_{i=1}^N h_i\cdot g_i\cdot\LT{p_{\beta^{(i)},b_i}}, \]
\ie, $f\in\LI{G_m}$.\\
Next, let $f\in I_0$. Then, with the same argument as for the $f_i$ above, there exists a $p_{\gamma, c}\in G_m$ such that $\LT{p_{\gamma, c}}|\LT{f}$. This shows that $G_m$ is a strong Gr\"{o}bner basis.\\
It remains to show that $G_m$ is minimal. To this end, pick two pairs $(\alpha,a), (\beta,b)\in S_m$ such that $a\vars^{\alpha}|b\vars^{\beta}$. Then $a\divm b,\:a\divz m,\:b\divz m,$ and $\alpha\preceq\beta$. We need to prove that $a=b$ and $\alpha=\beta$. Computing in $\Z$, take a prime factor $q$ of $b$ and $k\geq 1$ maximal such that $q^k\divz b$. Suppose $q^k\ndivz a$. Then $a\alpha!$ would have at least one less factor $q$ in its prime factorization than $b\alpha!$. But since $m\divz a\alpha!$, we then had $m\divz b/q\cdot\alpha!\divz b/q\cdot\beta!$, and $b$ would not be minimal in $(\beta,b)\in S_m$. We conclude that $b\divz a$. We write this as $a=d\cdot b$ for some $d\divz m$. Now $a\divm b$, that is, $m\divz a\cdot c-b$ for some $c$. Putting things together we get $bd=a\divz m\divz bcd - b = b(cd-1)$. Hence $d\divz (cd-1)$ which can only hold for $d=1$, implying $a=b$. But then we must also have $\alpha=\beta$, since otherwise $\beta$ would not be minimal in $(\beta,b)\in S_m$.
\end{proof}
We now show that leading terms of minimal strong Gr\"{o}bner bases of\linebreak $I_0\subset\Zmod{m}[x_1,x_2,\ldots,x_n]$ are unique, up to multiplication by units of $\Zmod{m}$. We prove this result as a consequence of a more general statement for ideals over arbitrary commutative rings with $1$ that has, to our knowledge, not been stated before. (Note the similar statement in the field case; see e.g.~Proposition 1.8.4 in \cite{adams}.)
\begin{thm}
a) Let $G, F$ be two minimal strong Gr\"{o}bner basis of an arbitrary ideal $I\subset\cR[x_1,x_2,\ldots,x_n]$, where $\cR$ is any commutative ring with $1$. Then $|G|=|F|$, and the sets of leading terms in $G$ and $F$ coincide up to multiplication by elements of $\cR$, \ie,
\[ \forall\:g\in G\:\:\:\exists\:f\in F\:\:\:\exists\:c\in\cR\:\:\:\LT{g} = c\cdot\LT{f}.\:\:\:\:(*) \]
b) In the case of $\cR=\Zmod{m}$ and $I=I_0$, the ring elements $c$ in $(*)$ can be chosen to be units of $\Zmod{m}$.
\end{thm}
Note that the second statement holds for any ideal, if the ring $\cR$ is a domain.
\begin{proof}
a) Starting with the proof of $(*)$, we pick any $g\in G\subset I$. Then, by strongness of $F$, there is some $f\in F$ such that $\LT{f}|\LT{g}$. Vice versa, by strongness of $G$, there must be some $g'\in G$ such that $\LT{g'}|\LT{f}$. Therefore, $\LT{g'}|\LT{f}|\LT{g}$, which implies $g = g'$, by minimality of $G$. But then the leading monomials $\LM{f}$ and $\LM{g}$ must also coincide, yielding the desired relation between $\LT{f}$ and $\LT{g}$.\\
Similar to the previous argument, it is easy to see that no two distinct leading terms in $F$ can fulfil a relation $(*)$ with the same leading term in $G$, and vice versa. This implies the equality $|\{\LT{g}~|~g\in G\}|=|\{\LT{f}~|~f\in F\}|$ which clearly amounts to $|G|=|F|$, by the minimality of $G$ and $F$.\\
b) We first choose $G=G_m$ to be the explicitely given Gr\"{o}bner basis, and $F$ any other minimal strong Gr\"{o}bner basis of $I_0\subset\Zmod{m}[x_1,x_2,\ldots,x_n]$. Consider a relation as in $(*)$, \ie, $b\cdot\vars^{\beta}=c\cdot a\cdot\vars^{\alpha}$ , where $(\beta,b)\in S_m$ and $a\cdot\vars^{\alpha}$ denotes the leading term of some $f\in F$. Then $b=a\cdot c$ mod $m$, in other words $m\divz ac-b$. Now let $\tilde{a}:=\ggt{a,m}$ be the maximum portion of $a$ that divides $m$, that is, $a=\tilde{a}\cdot u$, where $\ggt{u,m}=1$ which is equivalent to $u$ being a unit in $\Zmod{m}$. Since $\tilde{a}\divz m\divz ac-b$, we obtain $\tilde{a}\divz b$.\\
We want to show $\tilde{a}=b$, so for a contradiction let us assume $\tilde{a}<b$. $f\in F\subset I_0$ implies $m\divz a\alpha!$ by Lemma 3.2, hence $m\divz \tilde{a}\alpha!=\tilde{a}\beta!$, as the factors in $a/\tilde{a}$ do not affect divisibility by $m$ and since obviously $\alpha=\beta$. But this means that we could replace $b$ by the smaller $\tilde{a}$ and still preserve the condition $m\divz \tilde{a}\beta!$. This contradicts the minimality of $b$ in $(\beta,b)\in S_m$. Hence $\tilde{a}=b$.\\
We thus arrive at the claimed relation $u\cdot b\vars^{\beta}=a\vars^{\alpha}$, and $c$ can be replaced by the unit $u^{-1}\in(\Zmod{m})^*$.\\
We have shown that we can relate the leading terms of any minimal strong Gr\"{o}bner basis $F$ of $I_0\subset\Zmod{m}[x_1,x_2,\ldots,x_n]$ to the leading terms in $G_m$ by units. By transitivity, we can now clearly also relate the leading terms of any two minimal strong Gr\"{o}bner bases by units. This concludes the proof.
\end{proof}
Note that an arbitrary factor $c$, relating two leading terms, need not necessarily be a unit. For example, consider the polynomial $f(x,y)=3(x-1)(x-2)\cdot$\linebreak $(y-1)(y-2)\in G_{12}$. We may switch to another minimal strong Gr\"{o}bner basis of $I_0\subset\Zmod{12}[x,y]$, simply by replacing $f(x,y)$ by $f'(x,y)=9(x-1)(x-2)\cdot$\linebreak $(y-1)(y-2)$. Note that over $\Zmod{12}$ the ideals $\ideal{f}$ and $\ideal{f'}$ are identical. Thus, $G_m\setminus\{f\}\cup\{f'\}$ must still be a minimal strong Gr\"{o}bner basis. Now obviously $\LT{f'}=3\cdot\LT{f}$, but $3$ is not a unit in $\Zmod{12}$.\\[1ex]
We point out that minimal strong Gr\"{o}bner bases are in general not unique. This is due to the fact that we only consider leading terms and do not require tail reduction here. For example, in the case of the ideal $I_0$, we can easily modify the basis $G_m$ and still obtain a minimal strong Gr\"{o}bner basis. To this end, we may pick two elements $f,g\in G_m$ with $\LM{g}<\LM{f}$ and replace $f$ by $f+g$.\\[1ex]
Let us once again take a look at the complexity of $G_m$, that is, the size $|G_m|$ as a function of the number of variables $n$. The discussion that followed Example 3.4 already made clear that $|G_m|$ is only linear in $n$, when all prime factors of $m$ are mutually distinct. In the general case when $m=q_1^{e_1}\cdot q_2^{e_2}\cdots q_k^{e_k}$ with some $e_j>1$, the construction is combinatorially more complex. However, based on the following investigation for the practically relevant case $m=q^k$, we conjecture that for fixed $m$ the size of $G_m$ is always of polynomial order in $n$.\\
Since we are interested in the asymptotic behaviour of $|G_m|$ for large $n$, we may assume that $n$ is much larger than $m=q^k$. We can decompose $G_m$ into the disjoint union
\begin{align*}
G_m = & \bigcup_{0\leq j<k} G_m^{(j)},\:\:\text{where}\\
G_m^{(j)} := & \{q^j\cdot(x_i-1)\cdots(x_i-(k-j)q)\:~|~\:1\leq i\leq n\}\\
\cup & \{q^j\cdot(x_{i_1}-1)\cdots(x_{i_1}-s_1q)(x_{i_2}-1)\cdots(x_{i_2}-s_2q)\:~|\\
& \:\:\:\:1\leq i_1,i_2\leq n;i_1\neq i_2;1\leq s_1,s_2;s_1+s_2=k-j\}\\
& \cdots\\
\cup & \{q^j\cdot(x_{i_1}-1)\cdots(x_{i_1}-q)(x_{i_2}-1)\cdots(x_{i_2}-q)\cdots\\
& \:\:\:\:(x_{i_{k-j}}-1)\cdots(x_{i_{k-j}}-q)   \:~|~\:1\leq i_u\leq n;i_u\neq i_v\:\text{for}\:u\neq v\},\\
\end{align*}
that is, in $G_m^{(j)}$ we have the constant coefficient $q^j$, and we have polynomials in $1$ up to $k-j$ variables. With $h_j:=|G_m^{(j)}|$, we obtain the very rough estimates
\begin{align*}
h_j & \leq n+\binom{n}{2}\cdot k^1+\cdots+\binom{n}{k-j}\cdot k^{k-j-1}= \sum_{l=1}^{k-j}\binom{n}{l}\cdot k^{l-1}\leq \binom{n}{k}\cdot k^k,\\
h_j & \geq \binom{n}{k-j}.
\end{align*}
For $h:=|G_m|=\sum_{0\leq j<k}h_j$ we thus get
\[ \binom{n}{k}\leq \sum_{j=0}^{k-1}\binom{n}{k-j}\leq h\leq k\cdot\binom{n}{k}\cdot k^k= \binom{n}{k}\cdot k^{k+1},\]
and $h=|G_m|$ is of polynomial order of degree $k$ in the number of variables $n$.

\subsection{Computing the Reduced Normal Form\\
of a Polynomial}
After we have given a minimal strong Gr\"{o}bner basis of $I_0\subset\Zmod{m}[x_1,x_2,\ldots,x_n]$, we shall now turn to computing representatives of the residue classes in\linebreak $\left(\Zmod{m}[x_1,x_2,\ldots,x_n]\right)/I_0$. When we impose certain bounds on the coefficients of all monomials, these representatives are unique:
\begin{prop}\label{prop_unique_repr}
Every residue class~$\bar{f}\in\left(\Zmod{m}[x_1,x_2,\ldots,x_n]\right)/I_0$ has a\linebreak unique representative~$f\in\Zmod{m}[x_1,x_2,\ldots,x_n]$ of the form
\[f = \sum_{\alpha\in\{0,1,\ldots,m-1\}^n}a_\alpha\monom^\alpha,\:\:\text{where}\:\:0\leq a_{\alpha} < \frac{m}{\ggt{m,\alpha!}},\:\text{for all}\:\alpha.\]
\end{prop}
Note that, whenever $m\divz\alpha!$, the given bound forces $a_{\alpha}$ to be zero.
\begin{proof}
Let $f\in\Zmod{m}[x_1,x_2,\ldots,x_n]$ be an arbitrary polynomial. Suppose $f$ containes a monomial $a\vars^{\alpha}$ for which $a\geq c:=\frac{m}{\ggt{m,\alpha!}}$.
Due to division with remainder of $a$ by $c$ in $\Z$, we obtain $a=k\cdot c + r$ for some $k\in\{1, 2,\ldots\}$, and $0\leq r < c$. Now, $m\divz\frac{m\alpha!}{\ggt{m,\alpha!}}$. In other words, $m\divz c\alpha!$, and $p_{\alpha,c}\in I_0$ by Lemma 3.1.\\
As a consequence, $f$ and $f':=f-k\cdot p_{\alpha,c}$ lie in the same residue class. Moreover, the coefficient of $\vars^{\alpha}$ in $f'$ is $a-k\cdot c=r$, for which the claimed bound holds. Since we have a global order on the monomials, we need only finitely many repetitions of the presented reduction step, in order to arrive at a polynomial $g$ which also lies in the residue class of $f$, and the coefficients of which all satisfy the required bound condition.\\
For proving uniqueness of the constructed representative, assume we have two representatives $f_1,f_2$ of the residue class of $f$, realising all coefficient bounds. Then, by defining either $g:=f_1 -f_2$ or $g:=f_2 -f_1$, we obtain a polynomial $g\in I_0$ with $\LT{g}=a\vars^{\alpha}$ and $0\leq a < \frac{m}{\ggt{m,\alpha!}}$. By Lemma 3.2, we know that $m\divz a\alpha!$.\\
We need to show that $a=0$; so for a contradiction, let us assume that $a>0$. With $b:=\ggt{m,a}$ we still have $m\divz b\alpha!$, \ie, $\frac{m}{b}\divz\alpha!$. Then also $\frac{m}{b}\divz\ggt{m,\alpha!}$ which implies $m\divz b\cdot\ggt{m,\alpha!}$. But $b\cdot\ggt{m,\alpha!}\leq a\cdot\ggt{m,\alpha!}<m$, yielding the desired contradiction.
\end{proof}
As an immediate consequence, we can count the number of polynomial functions which is the same as the number of residue classes in $\left(\Zmod{m}[x_1,x_2,\ldots,x_n]\right)/I_0$:
\begin{cor}
The number of polynomial functions~\(\left(\Zmod{m}\right)^n\to\Zmod{m}\) is given by
\[N=\prod_{\alpha\in\{0,1,\ldots,m-1\}^n}\frac{m}{\ggt{m,\alpha!}}.\]
\end{cor}
In comparison, the number of all functions~\(\left(\Zmod{m}\right)^n\to\Zmod{m}\) equals
\[m^{\left(m^n\right)}=\prod_{\alpha\in\{0,1,\ldots,m-1\}^n}m=N\cdot\prod_{\alpha\in\{0,1,\ldots,m-1\}^n}\ggt{m,\alpha!}.\]
\begin{table}[h]
\centerline{\begin{tabular}{|c|c|c|}
  \hline
  $\Zmod{m}\longrightarrow\Zmod{m}$& No. of functions & No. of polynomial functions\\\hline
  $m = 2^2$ & $256$ & $64$\\
  $m = 2^8$ & $10^{616}$ & $10^{16}$\\
  $m = 2^{16}$ & $10^{315652}$ & $10^{52}$\\
  $m = 2^{32}$ & $10^{41373247567}$ & $10^{184}$\\
  \hline
\end{tabular}}
\end{table}
Hence, if~$m$ is not prime, there are much fewer polynomial functions\linebreak $(\Zmod{m})^n\rightarrow\Zmod{m}$ than functions. This has the consequence that not every problem which can be modelled by functions, like problems coming from formal verification, can be modelled by polynomials over $\Zmod{m}$ (cf.~\cite{datacorr} where, nevertheless, polynomial ideals over $\Zmod{2^k}$ have been used successfully).\\[1ex]
Following the idea in the proof of Proposition 3.6, we are able to present a very fast algorithm for computing the reduced normal form, that is, the unique representative of a residue class in the ring $\Zmod{m}[x_1,x_2,\ldots,x_n]$ module $I_0$. (see \cite{equi_check} for $\Zmod{2^k}$):

\begin{algorithm}
\caption{Reduced normal form in $\Zmod{m}[x_1,x_2,\ldots,x_n]$ with respect to $\vanideal$}
\label{NFzero}
\begin{algorithmic}
\REQUIRE $f\in \Zmod{m}[x_1,x_2,\ldots,x_n]$ a polynomial, $>$ any monomial order on $\Zmod{m}[x_1,x_2,\ldots,x_n]$
\ENSURE $h$ the reduced normal form of $f$ with respect to \vanideal
\STATE $h := 0$
\WHILE{$f\neq 0$}
  \STATE $a\monom^\alpha := \LT{f}$
  \STATE $c:=\frac{m}{\ggt{m, \alpha!}}$
  \STATE solve $a = k\cdot c+r$ with $k\in\N$ and $0\leq r < c$
  \STATE $h := h + r\monom^\alpha$
  \STATE $f := f - k\cdot p_{\alpha, c} - r\monom^\alpha$
\ENDWHILE
\RETURN $h$
\end{algorithmic}
\end{algorithm}
Note that the algorithm makes sure that $f+h$ will always represent the same residue class, as $p_{\alpha,c}\in I_0$. Since initially $h=0$, this class must be the residue class of $f$. After termination, which is ensured by the global order, $h$ consists only of terms with appropriate coefficient bound, \ie, $h$ must be the unique representative as given in Proposition 3.6.

\subsection{Computing Minimal Strong Gr\"{o}bner Bases\\
over Different Rings $\Zmod{m}$}
The simple structure of minimal strong Gr\"{o}bner bases provides us with a recursive means to construct $G_m$ from bases for smaller $m$. We are especially interested in computing $G_M$ from the elements of the already computed set $G_m, m>1$, where $M = q\cdot m$ with $q$ a prime number. Let $q$ be the maximal prime number dividing $M$, then the following (not necessarily pairwise disjoint) decomposition is being used in the below Algorithm {\it RecComp} to compute $G_M$:
\begin{prop}\label{prop_decomp}
Suppose $M=q\cdot m,\;m > 1$, where $q$ is the largest prime factor of $M$. Then, with the above notations, $G_M$ can be decomposed as follows.
\begin{align*}
G_M & = H_1\;\cup\;H_2\;\cup\;H_3,\;\;\mbox{where}\\
H_1 & := \{ p_{\alpha,a} ~|~ p_{\alpha,a}\in G_m, (\alpha,a)\in S_M \},\\
H_2 & := \{ p_{\alpha,aq} ~|~ p_{\alpha,a}\in G_m, (\alpha,aq)\in S_M \},\;\;\mbox{and}\\
H_3 & := \{ p_{\alpha+\beta,b} ~|~ \exists\:p_{\alpha,a}\in G_m,\:\exists\:\beta\in B(\alpha,a,q)\:\:\exists\: b\divz M\: :\:\:(\alpha+\beta,b)\in S_M \},
\end{align*}
with $B(\alpha,a,q)$ denoting the set of all $\beta\succ (0,0,\ldots,0)$ such that $(\alpha+\beta)!$ contains exactly one more prime factor $q$ than $a\alpha!$.
\end{prop}
This decomposition says that we may already directly find elements of $G_M$ in $G_m$. Or, secondly, we may build an element of $G_M$ by multiplying an element of $G_m$ by $q$. Besides altering the coefficient only, we can also try to enlarge the exponent vector of some $p_{\alpha,a}\in G_m$ such that the new exponent factorial $(\alpha+\beta)!$ contains one more prime factor $q$ than $a\alpha!$. However, enlarging the exponent may introduce many more divisors of $M$, so that in general we need to adjust the coefficient. It is easy to see that once a suitable $\beta$ is found, we can set $b=\frac{M}{\ggt{M,(\alpha+\beta)!}}$. The search for suitable $\beta$ can obviously be limited to the set defined by the condition $\beta\preceq (q,q,\ldots,q)$, that is, we know a finite superset of $B(\alpha,a,q)$.
\begin{proof}
It is immediately clear that $G_M\supset H_1\cup H_2\cup H_3$, since each of the definitions of the $H_i$ enforces that the respective index belongs to $S_M$.\\
For the other inclusion we choose an arbitrary $p_{\alpha,a}\in G_M$ and are going to show that this polynomial lies in at least one of $H_1,H_2$ and $H_3$.\\
If $p_{\alpha,a}\in G_M$ also belongs to $G_m$, then $p_{\alpha,a}\in H_1$. Thus in the following, let us assume that $p_{\alpha,a}\in G_M\setminus G_m$.\\
Suppose $q\divz a$, then $a=q\cdot a'$ for some $1\leq a'<m$ dividing $m$. We obtain the membership $p_{\alpha,a'}\in G_m$, i.e.~$p_{\alpha,a}=p_{\alpha,qa'}\in H_2$.\\
The remaining case is $q\not\divz a$; hence $q\divz\alpha_1!\cdots\alpha_n!$, with $n$ denoting the number of variables as usual. Without loss of generality $q$ then divides $\alpha_1!$. This means that there is some maximal $k\geq 1$ satisfying $k\cdot q\leq\alpha_1$. We then define the following objects:\\
$\begin{array}{llcl}
\quad\quad\quad\quad & \beta_1 & :=      & k\cdot q - 1,\\
\quad\quad\quad\quad & \beta   & :=      & (\beta_1,\alpha_2,\ldots,\alpha_n),\\
\quad\quad\quad\quad & b       & :=      & \frac{m}{gcd(m, \beta!)},\\
\quad\quad\quad\quad & \gamma  & \preceq & \beta\;\;\mbox{minimal, such that}\;gcd(m,\gamma!)=gcd(m,\beta!),
\end{array}$\\
where there may be numerous choices for $\gamma$.\\
Note first that $m=\frac{m}{gcd(m,\beta!)}\cdot gcd(m,\beta!)=b\cdot gcd(m,\gamma!)\;\divz\;b\cdot\gamma!$. The above construction includes the case $b=m$, and we shall consider it below. If $b<m$, then - due to the construction - both $b$ and $\gamma$ are minimal in the sense of the definition of $S_m$, thus $(\gamma,b)\in S_m$. But since $\gamma\preceq\beta\prec\alpha$, we have constructed some $p_{\gamma,b}\in G_m$ that gives rise to $p_{\alpha,a}$ by means of strictly enlarging the multi-exponent. Hence $p_{\alpha,a}\in H_3$, unless $b=m$.\\
Last, we consider the setting $b=m$. By the definition of $b$, this implies that $\beta!$ and $m$ are coprime. Then $k=1$, for if $k\geq 2$ then $\beta_1!$ would mention at least all prime numbers up to $q$ including $q$. But $q$ is the maximal of all prime factors of $M$ and hence greater than or equal to any prime factor of $m$. So this would contradict the afore mentioned coprimality, and we conclude that $k=1$.\\
Therefore $\beta_1!=1\cdot 2\cdots (q-1)$ and $m$ are coprime, implying that $m$ and $M$ must be powers of $q$. We have $q\not\divz a$ which means that in $a\alpha!$ all prime factors $q$ must come from $\alpha$ (and $a=1$). Then, by the minimality of $\alpha$, $\alpha=(k_1\cdot q,k_2\cdot q,\ldots,k_n\cdot q)$, for some $k_i\in\N_0$. And since $m>1$, either one of the $k_i$ is at least $2$, or there are at least two non-zero $k_i$'s. In either situation, we can immediately define some $\beta\prec\alpha$ which is not the zero multi-index, such that $\beta!$ has exactly one factor $q$ less than $\alpha!$. Then $(\beta,1)\in S_m$, and again $p_{\alpha,a}\in H_3$.
\end{proof}
The following examples are numbered according to the order in the above decomposition, that is, 1.~gives an example for the set $H_1$ and so on. (The number of variables, $n$, equals $2$.)
\begin{example}\hspace*{1ex}\\
\vspace*{-3ex}
\begin{enumerate}
\item $G_3\subset G_6$, since $3!=6$ already contains all necessary factors; see Example 3.4 (and the remark regarding $k = 1$) to recall the elements of $G_3$.
\item With $q$ any prime, we have $p_{(3,0),2}\in G_{12}$ and $p_{(3,0),2\cdot q}\in G_{12\cdot q}$.
\item We have $6(x-1)(x-2)(y-1)(y-2)\in G_{24}$. We try to construct an element in $G_{24\cdot 3}$ by enlarging the product of $x$ and $y$ terms. Since\linebreak $6\cdot 2!\cdot 2!$ contains one prime factor $3$, we try to move to the target product $(x-1)(x-2)(x-3)(y-1)(y-2)(y-3)$ which realizes one more factor $3$ because $3^2\divz 3!\cdot 3!$. Now $b=\frac{72}{\ggt{72,3!\cdot 3!}}=2$ and hence $2(x-1)(x-2)\cdot$\linebreak $(x-3)(y-1)(y-2)(y-3)\in G_{72}$.
\end{enumerate}
\end{example}
Note that in Proposition \ref{prop_decomp}, we need $q$ to be the {\it largest} and not just any prime factor of $M$. For a counterexample, consider $m=3,\;q=2$, hence $M=6$, and the univariate case, i.e.~$n=1$. Then one easily checks that $H_1=G_3, H_2=H_3=\emptyset$ but $H_1\cup H_2\cup H_3=G_3\neq G_6$.\\[1ex]
The above decomposition of $G_M$, and the structure of $G_q$ for a prime $q$ as discussed in Example 3.4, give rise to the following algorithm.\pagebreak

\begin{algorithm}[h]\label{algo_recursive}
\caption{RecComp($M$), Recursive computation of $G_M$}
\label{RecursGm}
\begin{algorithmic}
\REQUIRE $M\in\{2,3,\ldots\}$
\ENSURE $G_M$
\STATE $q :=$ the largest prime factor of $M$
\IF{$M=q$}
  \STATE $A: =\{ q\cdot e_i ~|~ 1\leq i\leq n\}$, where the $e_i$ are the unit vectors in $\N^n$
  \STATE $G := \{ p_{\alpha,1} ~|~ \alpha\in A \}$
\ELSE
  \STATE $m:=M/q$
  \STATE $H:=$RecComp($m$)
  \STATE $G:=\{\:\}$
  \FORALL{$p_{\alpha,a}\in H$}
    \IF{$(\alpha,a)\in S_M$}
      \STATE $G:=G\cup\{p_{\alpha,a}\}$
    \ELSE
      \STATE $G:=G\cup\{p_{\alpha,a\cdot q}\}$
      \FORALL{$\beta\in B(\alpha,a,q)\subset\{\beta~|~(0,0,\ldots,0)\prec\beta\preceq (q,q,\ldots,q)\}$}
        \STATE $b:=\frac{M}{\ggt{M,(\alpha+\beta)!}}$
        \STATE $G:=G\cup\{p_{\alpha+\beta,b}\}$
      \ENDFOR      
    \ENDIF
  \ENDFOR
\ENDIF
\RETURN $G$
\end{algorithmic}
\end{algorithm}

\bibliographystyle{plain}
\bibliography{groebner}

\begin{thebibliography}{10}

\bibitem{adams}
W.~Adams and P.~Loustaunau.
\newblock {\em \em An introduction to Gr{\"o}bner bases.}
\newblock (Graduate studies in mathematics) AMS, 2003.

\bibitem{newdevel}
M.~Brickenstein, A.~Dreyer, G.-M. Greuel, M.~Wedler, and O.~Wienand.
\newblock New developments in the theory of groebner bases and applications to
  formal verification.
\newblock {\em Journal of Pure and Applied Algebra}, 2008.

\bibitem{carlitz}
L.~Carlitz.
\newblock Functions and polynomials (mod $p^n$).
\newblock {\em Acta Arith.}, 9:67--78, 1964.

\bibitem{singsys}
G.-M. Greuel, G.~Pfister, and H.~Sch{\"o}nemann.
\newblock {\em {\sc Singular} 3.1.0 --- {A} computer algebra system for
  polynomial computations}.
\newblock 2009.
\newblock http://www.singular.uni-kl.de.

\bibitem{halbeisen}
L.~Halbeisen, N.~Hungerb{\"u}hler, and H.~L{\"a}uchli.
\newblock Powers and polynomials in $\mathbb{Z}/m$.
\newblock {\em Elem. Math.}, 54:118–--129, 1999.

\bibitem{gen_smarand}
N.~Hungerb{\"u}hler and E.~Specker.
\newblock A generalization of the smarandache function.
\newblock {\em Integers: Electronic J. Combinatorial Number Theory}, 6, 2006.

\bibitem{kempner}
A.~J. Kempner.
\newblock Polynomials and their residue systems.
\newblock {\em Amer. Math. Soc. Trans.}, 22:240–--266, 1921.

\bibitem{equi_check}
N.~Shekhar, P.~Kalla, F.~Enescu, and S.~Gopalakrishnan.
\newblock Equivalence verification of polynomial datapaths with fixed-size
  bit-vectors using finite ring algebra.
\newblock In {\em ICCAD '05: Proceedings of the 2005 IEEE/ACM International
  conference on Computer-aided design}, pages 291--296, Washington, DC, USA,
  2005. IEEE Computer Society.

\bibitem{singmaster}
D.~Singmaster.
\newblock On polynomial functions (mod m).
\newblock {\em Journal of Number Theory}, 6:345--352, 1974.

\bibitem{smarandache}
F.~Smarandache.
\newblock A function in the number theory.
\newblock {\em Analele Univ.~Timisoara, Fascicle 1}, XVIII:79--88, 1980.

\bibitem{ollisphd}
O.~Wienand.
\newblock {Ph.D.} thesis.
\newblock Kaiserlautern, Germany, In prepration, 2009.

\bibitem{datacorr}
Oliver Wienand, Markus Wedler, Dominik Stoffel, Wolfgang Kunz, and Gert-Martin
  Greuel.
\newblock An algebraic approach for proving data correctness in arithmetic data
  paths.
\newblock In {\em CAV '08: Proceedings of the 20th international conference on
  Computer Aided Verification}, pages 473--486, Berlin, Heidelberg, 2008.
  Springer-Verlag.

\end{thebibliography}

\end{document}